\documentclass[a4paper,11pt]{amsart}
\usepackage{lmodern}
\usepackage{amsmath,amsfonts,amssymb,mathrsfs,amsthm}

\newtheorem{theorem}[subsection]{Theorem}
\newtheorem{lemma}[subsection]{Lemma}
\newtheorem{cor}[subsection]{Corollary}

\newtheorem{remark}[subsection]{Remark}

\newcommand{\dist}{{\mathop{\rm dist}}}

\newcommand{\A}{{\mathcal A}}
\newcommand{\B}{{\mathcal B}}
\newcommand{\C}{{C}}

\newcommand{\E}{{\mathcal E}}
\newcommand{\hc}{{\mathcal H}}
\newcommand{\hk}{{\mathcal K}}

\newcommand{\N}{{\mathcal N}}

\newcommand{\norm}[1]{\left\lVert#1\right\rVert}

\newcommand{\bF}{{\mathbb{F}}}
\newcommand{\bR}{{\mathbb{R}}}
\newcommand{\bC}{{\mathbb{C}}}
\newcommand{\bN}{{\mathbb{N}}}

\newcommand{\tr}{{\mathrm{tr }}}

\newcommand{\bra}{\langle}
\newcommand{\ket}{\rangle}
\newcommand{\tens}[1]{\mathbin{\mathop{\otimes}\limits_{#1}}}
\begin{document}


\title{Best approximations, distance formulas and orthogonality in $C^*$-algebras}

\author[Grover]{Priyanka Grover}
\author[Singla]{Sushil Singla}
\address{Department of Mathematics, Shiv Nadar University, NH-91, Tehsil Dadri, Gautam Buddha Nagar, U.P. 201314, India.}
\email{priyanka.grover@snu.edu.in, ss774@snu.edu.in}

\maketitle

\begin{abstract} 
For a unital $C^*$-algebra $\A$ and a subspace $\B$ of $\A$, a characterization for a  best approximation to an element of $\A$ in $\B$ is obtained. As an application, a formula for the distance of an element of $\A$ from $\B$ has been obtained, when  a best approximation of that element to $\B$ exists. Further, a characterization for Birkhoff-James orthogonality of an element of a Hilbert $C^*$-module to a subspace is obtained.
\end{abstract}

\setlength{\parindent}{0pt}
\setlength{\parskip}{1.6ex}

\pagestyle{headings}

\section{Introduction}
Let $\A$ be a unital $C^*$-algebra over $\bF(=\bR$ or $\bC)$ with the identity element  $1_{\A}$. The $C^*$-subalgebras of $\A$  are assumed to contain $1_{\A}$. For $a\in \A$ and $\B$ a subspace of $\A$, $\dist(a,\B)$ denotes $\inf\{\|a-b\|:b\in \B\}$. An element $b_0\in \B$ is said to be \emph{a best approximation to $a$ in $\B$} if $\|a-b_0\| = \dist(a,\B)$. It is a well known fact that $b_0$ is a  best approximation to $a$ in $\B$ if and only if there exists a functional $\psi\in\A^*$ such that $\psi(a-b_0)=\dist(a,\B)$ and $\psi(b)=0$ for all $b\in\B$ (see \cite[Theorem 1.1]{singer}).

Let $(\C(X),\|\cdot\|_\infty)$ be the $C^*$-algebra of real or complex continuous functions on a compact Hausdorff space $X$, where $\|f\|_\infty=\sup_{x\in X}|f(x)|.$ It was proved in Theorem 1.3 of \cite{singer} that if $f\in\C(X)$ and $\B$ is a subspace of $\C(X)$, then $g$ is a best approximation to $f$  in $\B$ if and only if there exists a regular Borel probability measure $\mu$ on $X$ such that the support of $\mu$ is contained in the set $\{x\in X: |(f-g)(x)| = \|f-g\|_\infty\}$ and $\int\limits_X \overline{(f-g)}h \,d\mu = 0 \text{ for all } h\in\B$. The condition that the support of $\mu$ is contained in the set $\{x\in X: |(f-g)(x)| = \|f-g\|_\infty\}$ is equivalent to $\int\limits_X |f-g|^2 \,d\mu = \|f-g\|_\infty^2$.



A \emph{positive linear map} from $\A$ to another $C^*$-algebra $\A_0$ is a linear map that maps positive elements of $\A$ to positive elements of $\A_0$. For $\bF=\bC$, a \emph{state} on $\A$ is a positive linear functional $\phi$ on $\A$ such that $\phi(1_{\A}) =1$. For $\bF=\bR$, an additional requirement for $\phi$ to be a state is that $\phi(a^*) = \phi(a)$ for all $a\in\A$. Let $\mathcal S_{\A}$ denotes the set of states on $\A$. Using Riesz Representation Theorem, the above characterization for best approximation in $\C(X)$ is equivalent to saying that there exists $\phi\in\mathcal S_{\C(X)}$ such that  \begin{equation} \phi(|f-g|^2) = \|f-g\|_\infty^2 \text{ and } \phi(\overline{(f-g)}h) = 0 \text{ for all } h\in\B.\end{equation} 

For $a\in \A$ and $\B$ a subspace of $\A$, $a$ is said to be \emph{Birkhoff-James orthogonal} to $\B$ (or \emph{$\B$-minimal})  if $\|a\|\leq \|a+b\|$ for all $b \in \B$. Note that this is equivalent to saying that $0$ is a best approximation to $a$ in $\B$. It was proved in Theorem 2 of \cite{williams} that $0$ is a best approximation to an element $a$ of a complex $C^*$-algebra $\A$ in $\bC 1_{\A}$ if and only if there exists $\phi\in \mathcal S_{\A}$ such that $\phi(a^*a)= \|a\|^2$ and $\phi(a)=0$.
Theorem 6.1 in \cite{rieffel 1} shows that if $\B$ is a $C^*$-subalgebra containing $1_{\A}$ of a complex $C^*$-algebra $\A$ and if $0$ is a best approximation to a Hermitian element $a$ of $\A$  in $\B$, then there exists $\phi\in S_{\A}$ such that $\phi(a^2)=\|a\|^2$ and $ \phi(ab+b^*a)=0$  for all $b\in\B$. In Proposition 4.10 of \cite{2013}, it was proved that for any elements $a$ and $b$ of a complex $C^*$-algebra $\A$, $0$ is a best approximation to $a$  in $\bC b$ if and only if there exists $\phi\in \mathcal S_{\A}$ such that $\phi(a^*a)= \|a\|^2$ and $\phi(a^*b)=0$. The main result of this article shows the existence of such a state for any element $a$ and for any subspace $\B$ of a $C^*$-algebra over $\bF$.  

\begin{theorem}\label{58} Let $a\in\A$. Let $\B$ be a subspace of $\A$.  Then $b_0$ is a best approximation to $a$ in $\B$ if and only if there exists $\phi\in\mathcal S_{\A}$ such that \begin{equation}\label{ss}\phi((a-b_0)^*(a-b_0)) = \|a-b_0\|^2 \text{ and } \phi(a^*b) = \phi(b_0^*b)\text{ for all }b\in\B.\end{equation}\end{theorem}

 For $\phi\in S_{\A}$ and $a_1,a_2\in \A$, define $\bra a_1|a_2\ket_{\phi} = \phi(a_1^*a_2)$. This is a semi-inner product on $\A$. Let $\|a_1\|_{\phi}=\bra a_1|a_1\ket_{\phi}^{1/2}$.  In this notation, the above theorem says that $b_0$ is a best approximation to $a$ in $\B$ if and only if there exists $\phi\in\mathcal S_{\A}$ such that
 $$\| a-b_0\|_{\phi}=\|a-b_0\| \text{ and } \bra a-b_0|b\ket_{\phi}=0 \text{ for all } b\in \B.$$

We note that \eqref{ss} is a Pythagoras theorem in the semi-inner product space $(\A,\bra \cdot|\cdot \ket_\phi)$. Consider the triangle with vertices $0, a, b_0$ in  $(\A, \bra\cdot|\cdot\ket_{\phi})$.
If $a\notin\B$, then \eqref{ss} gives that $\|a\|_{\phi}^2 = \|b_0\|_{\phi}^2 + \|a-b_0\|^2$ and $\bra a-b_0|b\ket_{\phi} = 0 \text{ for all } b\in\B$. 
If $\|b_0\|_{\phi} = 0$, then we have $\|a\|_{\phi} = \|a-b_0\|$. This means that the length of the base and the length of the perpendicular are $0$ and $\|a-b_0\|$, respectively. 
Suppose $\|b_0\|_{\phi} \neq 0$. Let $\theta_{\phi}^{a_1,a_2}=\cos^{-1}\left(\dfrac{\bra a_1|a_2\ket_{\phi}}{\|a_1\|_\phi \|a_2\|_\phi}\right)$ be the angle  between the vectors $a_1$ and $a_2$ in $(\A,\bra\cdot|\cdot\ket_{\phi})$, when $\|a_1\|_\phi, \|a_2\|_\phi\neq 0$. Then we have $\|a-b_0\|_{\phi} = \|a-b_0\|$ and $\theta_{\phi}^{a-b_0,b} = \pi/2$ for all $b\in\B$. In particular, the above triangle becomes a right angled triangle and the length of the perpendicular is $\|a-b_0\|$. 

As a consequence, we obtain a distance formula of an element $a\in \A$ from a subspace $\B$ of $\A$.

\begin{cor}\label{9999} Let $a\in\A$. Let $\B$ be a subspace of $\A$. If $b_0$ is a best approximation to $a$ in $\B$, then 
\begin{equation}
\dist(a,\B)^2 = \max\{\phi(a^*a) -\phi(b_0^*b_0) : \phi\in\mathcal S_{\A} \text{ and }\phi(a^*b) = \phi(b_0^*b) \text{ for all } b\in\B\}.\label{xx}
\end{equation} \end{cor}

A special case of the above corollary is the below result by Williams \cite{williams}. He proved that for $a\in \A$, \begin{equation}\dist(a, \mathbb C1_\A)^2 =\max\{\phi(a^*a)-|\phi(a)|^2:\phi \in \mathcal S_{\A}\}.\label{rie}\end{equation} 
See \cite[Theorem 3.10]{rieffel} for a different proof of \eqref{rie}. For $n\times n$ complex matrices, a different proof has also been given in \cite[Theorem 9]{audenaert}.

As a direct consequence of Theorem \ref{58}, we get the following characterization of Birkhoff-James orthogonality to a subspace in a $C^*$-algebra.

\begin{cor}\label{10} Let $a\in\A$. Let $\B$ be a subspace of $\A$. Then $a$ is Birkhoff-James orthogonal to $\B$  if and only if there exists  $\phi\in\mathcal S_{\A}$ such that $\phi(a^*a) = \|a\|^2 \text{ and }\phi(a^*b) =0 \text{ for all }b\in\B.$\end{cor}

Geometrically, this says that $a$ is Birkhoff-James orthogonal to $\mathcal B$ if and only if there exists $\phi\in \mathcal S_{\A}$ and a corresponding semi-inner product $\bra \cdot|\cdot\ket_{\phi}$ on $\A$ such that $\|a\|_{\phi} = \|a\|$ and $a$ is perpendicular to $\B$ in $(\A, \bra \cdot|\cdot\ket_{\phi})$. 


In Section \ref{proofs}, we give the proofs of Theorem \ref{58} and Corollary \ref{9999}. In Section \ref{applications}, we give some other applications of Theorem \ref{58}. In Theorem \ref{thm3.7}, we show that $0$ is a best approximation to $a$ in $\B$ if and only if $0$ is a best approximation to $a^*a$ in $a^*\B$. In Theorem \ref{8},  it is shown that for any element $a\in\A$ and a subspace $\B$ of $\A$, there exists a cyclic representation $(\hc, \pi, \xi)$ of $\A$ and a unit vector $\eta\in\hc$ such that  $\dist(a,\B) = \bra \eta | \pi(a)\xi \ket$ and $\langle  \eta | \pi(b)\xi \rangle=0$ for all $b\in\B$. In Theorem \ref{27}, a characterization for Birkhoff-James orthogonality of an element of a \emph{Hilbert $C^*$-module } to a subspace is given. It is proved that an element $e$ of a Hilbert $C^*$-module $\E$ over $\A$ is Birkhoff-James orthogonal to a subspace $\B$ of $\E$ if and only if there exists $\phi\in S_{\A}$ such that $\phi(\left<e,e\right>)= \|e\|^2 \mbox{ and } \phi(\left<e,b\right>) = 0$ for all $b\in\B$.  In \cite{rieffel}, it was desired to have the generalization of distance formula \eqref{rie} in terms of \emph{conditional expectations} from $\A$ to $\B$. In Section \ref{remarks}, we make some remarks on our progress towards obtaining this. Corollary 1.2, Corollary 1.3,  Theorem 3.5 and Equation \eqref{98} are mentioned  in the survey article \cite{GroverSingla}. We provide the complete details here.

\section{Proofs}\label{proofs}

Few notations are in order.  Let $\hc$ be a Hilbert space over $\bF$. The inner product is assumed to be conjugate linear in the first coordinate and linear in the second coordinate. Let $\mathscr B(\hc)$ be the $C^*$-algebra of bounded $\bF$-linear operators on $\hc$. The symbol $I$ denotes the identity in $\mathscr B(\hc)$. The triple  $(\hc, \pi, \xi)$ denotes a cyclic representation of $\A$ where $\|\xi\| =1$, $\pi: \A\rightarrow \mathscr B(\hc)$ is a $^*$-algebra map  satisfying $\pi(1_A) = I$ and closure of $\{\pi(a)\xi: a\in\A\}$ is $\hc$. 

\textit{Proof of Theorem \ref{58}} If $\phi$ is  a state such that \eqref{ss} holds, then for every $b\in\B$, 
\begin{eqnarray*}\|a-b_0\|^2 &=& \phi((a-b_0)^*(a-b_0))\\ 
&\leq&\phi((a-b_0)^*(a-b_0))  + \phi(b^*b)\\
&=&\phi((a-b_0 - b)^*(a-b_0 - b))\\
&\leq& \|a-b_0-b\|^2.\end{eqnarray*}

So $b_0$ is a best approximation to $a$ in $\B$. For the other side, first let us assume that $\A$ is a complex $C^*$-algebra. By the Hahn-Banach theorem,  there exists  $\psi\in \A^*$  such that $\|\psi\| = 1$, $\psi(a-b_0) = \dist(a,\B) = \|a-b_0\|$ and $\psi(b) = 0$ for all $b\in\B$.  By Lemma 3.3 of \cite{rieffel 1}, there exists a cyclic representation $(\hc, \pi, \xi)$ of $\A$ and a unit vector $\eta\in\hc$ such that  \begin{equation}
\psi(c) = \bra  \eta | \pi(c)\xi \ket \text{ for all }c\in\A.\label{eq1}\end{equation}
Now $\psi(a-b_0) = \bra \eta | \pi(a-b_0)\xi \ket = \|a-b_0\|$. So by using the condition for equality in Cauchy-Schwarz inequality, we obtain $\|a-b_0\|\eta = \pi(a-b_0)\xi$. Equation \eqref{eq1} gives $$\psi(c) = \dfrac{1}{\|a-b_0\|}\bra \pi(a-b_0)\xi | \pi(c)\xi \ket \text{ for all } c\in\A.$$ Therefore \begin{equation}
\bra \pi(a-b_0)\xi | \pi(a-b_0)\xi\ket = \|a-b_0\|^2 \label{eq2}
\end{equation} and \begin{equation}\bra  \pi(a-b_0)\xi | \pi(b)\xi \ket = 0 \text{ for all }b\in\B.\label{eq3}\end{equation}
Define $\phi\in \A^*$ as $\phi(c) = \bra  \xi | \pi(c)\xi \ket$. Then $\phi\in\mathcal S_{\A}$ and by \eqref{eq2} and \eqref{eq3}, we obtain \eqref{ss}.

Next, let $\A$ be a real $C^*$-algebra. Let $\A_c$ be the complexification of $(\A, \|\cdot\|)$ with the unique norm $\|\cdot\|_c$ such that $(\A_c, \|\cdot\|_c)$ is a $C^*$-algebra and the natural embedding of $\A$ into $\A_c$ is an isometry \cite[Corollary 15.4]{Goodearl}. From the above case, there exists $\psi\in S_{\A_c}$ such that $ \psi((a-b_0)^*(a-b_0)) = \|a-b_0\|^2$ and $\psi(a^*b) = \psi(b_0^*b) \text{ for all }b\in\B.$ Let $\phi = \text{Re } \psi|_{\A}$.  Then $\phi\in S_{\A}$, $\phi((a-b_0)^*(a-b_0)) = \|a-b_0\|^2$ and $\phi(a^*b) = \phi(b_0^*b) \text{ for all } b\in\B$.
\qed

Another proof of Theorem \ref{58}, in the case when $\A$ is a complex $C^*$-algebra, can be given as follows.  The importance of this approach is that it indicates that proving the theorem when $\B$ is a one dimensional subspace is sufficient.
Since $b_0$ is a best approximation to $a$ in $\B$, $0$ is a best approximation to $a-b_0$ in $\B$. So without loss of generality, we assume $b_0 =0$. 
For $b\in\B$,  we have $\|a\|\leq\|a+\lambda b\|$ for all $\lambda\in\bC$. By Proposition 4.1 of \cite{2013}, there exists $\phi_b \in\mathcal S_{\A}$ such that $\phi_b(a^*a) = \|a\|^2$ and $\phi_b(a^*b) = 0$. 
Let $\N = \{\alpha a^*a + \beta 1_{\A} + a^*b: \alpha, \beta\in\bC$, $b\in\B\}$, the subspace generated by $a^*a$, $1_{\A}$ and $a^*\B$. Define $\psi: \N \longrightarrow \bC$ as $\psi(\alpha a^*a + \beta 1_{\A} + a^*b) = \alpha\|a\|^2 + \beta$ for all $\alpha,\beta\in\bC$ and $b\in\B$. 
To see that $\psi$ is well defined, note that for any $b\in B$ we have $\phi_b(\alpha a^*a + \beta 1_{\A} + a^* b) = \alpha\|a\|^2 + \beta$. Since $\|\phi_b\| = 1$, we get \begin{equation}\label{93}|\alpha\|a\|^2 + \beta| \leq \|\alpha a^*a + \beta 1_{\A} + a^* b\|.\end{equation} Thus $\alpha a^*a + \beta 1_{\A} + a^* b = 0$ implies $\alpha\|a\|^2 + \beta = 0.$ Clearly $\psi$ is a linear map and equation \eqref{93} shows that $\|\psi\| \leq 1$. Since $\psi(1_{\A}) =1$, we have $\|\psi\| = 1$. 
By the Hahn-Banach theorem, there exists a linear functional $\phi: \A \rightarrow \bC$ such that $\|\phi\|= 1$ and $\phi|_{\N} = \psi$. Since $\|\phi\| = 1 = \phi(1_{\A})$, using Theorem II.6.2.5(ii) of \cite{Blackadar}, we get that $\phi\in\mathcal S_{\A}$. By definition, $\phi$ satisfies the required conditions.

\textit{Proof of Corollary \ref{9999}} Let $\phi\in\mathcal S_{\A}$ be such that  $\phi(a^*b) = \phi(b_0^*b)$ for all $b\in\B$. In particular we have $\phi(a^*b_0) = \phi(b_0^*b_0)$. So
$ \phi((a-b_0)^*(a-b_0)) = \phi(a^*a) - \phi(b_0^*b_0).$ 
 Since $\phi((a-b_0)^*(a-b_0)) \leq \norm{a-b_0}^2 = \dist(a, \B)^2$, we have
$$\phi(a^*a) - \phi(b_0^*b_0) \leq \dist(a, \B)^2.$$ 
This gives \begin{equation*}
\sup\{\phi(a^*a) -\phi(b_0^*b_0) : \phi\in \mathcal S_{\A}, \phi(a^*b) = \phi(b_0^*b) \text{ for all } b\in\B\}\leq \dist(a,\B)^2 . 
\end{equation*}
By Theorem \ref{58}, there exists $\phi\in S_{\A}$ such that 
\begin{equation*}\dist(a,\B)^2 =\phi(a^*a) -\phi(b_0^*b_0) \text{ and }\phi(a^*b) = \phi(b_0^*b) \text{ for all } b\in\B.\label{xxx}\end{equation*} This completes the proof.
\qed

\section{Applications}\label{applications}

An interesting fact arises out of Corollary \ref{10}, which is worth noting separately. 
\begin{theorem}\label{thm3.7}Let $a\in\A$. Let $\B$ be a subspace of  $\A$.  Then $a$ is Birkhoff-James orthogonal to $\B$ if and only if  $a^*a$ is Birkhoff-James orthogonal to $a^*\B$.
\end{theorem}
\begin{proof} First let $a$  is Birkhoff-James orthogonal to $\B$. Then by Corollary \ref{10}, there exists $\phi\in S_{\A}$ such that $\phi(a^*a)=\|a\|^2$ and $\phi(a^*b)=0$ for all $b\in \B$. So for $b\in \B$,  $\phi(a^*a+a^*b)=\|a\|^2$. Since $\|\phi\| = 1$, we get $\|a^*a\| = \|a\|^2\leq\|a^*a+a^*b\|.$ 
Conversely, suppose $a^*a$ is Birkhoff-James orthogonal to $a^*\B$, that is, $\|a^*a\|\leq\|a^*a+a^*b\|$ for every $b\in \B$. This implies $\|a\|^2\leq\|a^*\|\|a+b\|$ and thus $\|a\|\leq \|a+b\|$ for all $b\in \B$. \end{proof}

 We now show that Theorem 1 of \cite{2014} can also be proved using Corollary \ref{10}. We first prove the following lemma, which is of independent interest. The  proof of the lemma is along the same lines as a portion of the proof of Theorem 1 of \cite{Bhatia}.
For $u, v\in\hc$, $u\bar{\tens{}} v$ will denote the finite rank operator of rank one on $\hc$ defined as $u\bar{\tens{}} v(w) = \bra v | w\ket u$ for all $w\in\hc$.

\begin{lemma}\label{99} Let $A\in\mathscr B(\hc)$. Let $T$ be a positive trace class operator with $\|T\|_1 =1$ and $\tr(AT) = \|A\|$. Then there is an at most countable index set $\mathcal J$, a set of positive numbers $\{s_j : j\in\mathcal J\}$ and an orthonormal set $\{u_j: j\in\mathcal J\} \subseteq \text{Ker}(T)^{\bot}$  such that \begin{enumerate}
\item[(i)]  $\sum\limits_{j\in\mathcal J} s_j =1$ ,
\item[(ii)]$Au_j = \|A\|u_j$ for each $j\in \mathcal J$, \\
\item[(iii)] $T= \sum\limits_{j\in\mathcal J} s_j u_j\bar{\tens{}} u_j$.
\end{enumerate}\end{lemma} 
\begin{proof} Using Corollary 5.4  of \cite[Ch. II]{conway}, there exists a sequence of real numbers $s_1, s_2 \dots$ with orthonormal basis $\{u_1, u_2, \dots\}$ of $\text{Ker}(T)^{\bot}$ such that $T= \sum\limits_{i=1}^{\infty} s_i u_i\bar{\tens{}} u_i$. Since $T$ is positive, $s_i$ are non-negative. And $\|T\|_1 = 1$ implies  $\sum\limits_{i=1}^{\infty} s_i =1$.  Now $AT =\sum\limits_{i=1}^{\infty} s_i Au_i\bar{\tens{}}u_i$. Let $\mathcal J = \{i\in\bN : s_i\neq 0\}$. Then $\sum\limits_{j\in\mathcal J} s_j =1$ and $AT =\sum\limits_{j\in\mathcal J} s_j Au_j\bar{\tens{}}u_j$. {So $\tr(AT) = \sum\limits_{j\in\mathcal J} s_j \tr(Au_j\bar{\tens{}}u_j) = \sum\limits_{j\in\mathcal J} s_j \left\bra u_j|Au_j\right\ket$.}


Now
\begin{eqnarray*}
\|A\| &=& \tr(AT) = \sum\limits_{j\in\mathcal J} s_j \bra u_j|Au_j\ket ={\bigg|}\sum\limits_{j\in\mathcal J} s_j \bra u_j| Au_j \ket{\bigg|} \leq \sum\limits_{j\in\mathcal J} s_j{\bigg|}\bra u_j| Au_j\ket{\bigg|} \leq\sum\limits_{j\in\mathcal J} s_j\|Au_j\|\\ &\leq&\sum\limits_{j\in\mathcal J} s_j\|A\| =  \|A\|. \end{eqnarray*}
So\begin{eqnarray*} \sum\limits_{j\in\mathcal J} s_j{\bigg|}\bra u_j| Au_j\ket{\bigg|} = \sum\limits_{j\in\mathcal J} s_j\|Au_j\| =  \|A\|.\end{eqnarray*}
Therefore \begin{eqnarray*}0=\sum\limits_{j\in\mathcal J} s_j \left(\|A\| - {\bigg|}\bra u_j|Au_j\ket{\bigg|}\right) =  \sum\limits_{j\in\mathcal J} s_j\left(\|Au_j\| - {\bigg|}\bra u_j|Au_j\ket{\bigg|}\right).\end{eqnarray*}
Since $s_j>0$ for all $j\in \mathcal J$, we get \begin{equation}
\|A\| = {\bigg|}\bra u_j| Au_j\ket{\bigg|} = \|Au_j\| \text{ for all } j\in\mathcal J.\label{15}
\end{equation} 
By the condition of equality in Cauchy-Schwarz inequality, for every $j\in \mathcal J$ there exists $\alpha_j\in\bC$ such that $\alpha_j Au_j = u_j$. And using \eqref{15}, we get $Au_j = \|A\|u_j$. This completes the proof. 
\end{proof}

Let $\mathbb M_n(\bF)$ be the $C^*$-algebra of $n\times n$ matrices with entries in $\bF$. A \emph{density matrix} $A\in \mathbb M_n(\bF)$ is a positive element in $\mathbb M_n(\bF)$ with $\tr(A) = 1$.
A different proof of Theorem 1 in \cite{2014} follows.

\begin{theorem}\label{41}\cite[Theorem 1]{2014} \textit{Let $A\in \mathbb M_n(\bF)$. Let $m(A)$ be the multiplicity of the maximum singular value $\|A\|$ of $A$. Let $\B$ be a subspace of $\mathbb M_n(\bF)$. Then $A$ is Birkhoff-James orthogonal to $\B$ if and only if there exists a density matrix $T\in \mathbb M_n(\bF)$ of rank at most $m(A)$ such that $A^*AT = \|A\|^2T$ and $\tr(B^*AT) = 0$ for all $B\in \B$.}\end{theorem}
\begin{proof}By Corollary \ref{10}, there exists  a density matrix 
$T$ such that $\tr(A^*AT) = \|A\|^2$ and $\tr(B^*AT) = 0$ for all $B\in \B$. Using Lemma \ref{99}, there exists $s_1,\ldots, s_m$ and a set of orthonormal vectors $\{u_1, \ldots, u_m\}$ such that $\sum\limits_{j=1}^m s_j =1$, $A^*Au_j = \|A\|^2u_j$ for every $j=1,\ldots,m$ and $T= \sum\limits_{j=1}^m s_j u_j\bar{\tens{}} u_j$. Clearly $\text{rank } T\leq m\leq m(A)$ and $A^*AT = \|A\|^2T$. \end{proof}

It is worth noting that from the proof of Theorem \ref{41}, we get that $A^*AT = \|A\|^2T$ is equivalent to $\tr(A^*AT) = \|A\|^2$, where $A, T\in \mathbb M_n(\bF)$ and $T$ is a density matrix. This supplements Remark 1 of \cite{2014}.

Next we note that the idea of the proof of Theorem \ref{58} also proves the following generalization of Corollary 2.8 in \cite{2012}.

\begin{theorem}\label{8} Let $a\in\A$. Let $\B$ be a subspace of $\A$. Then there exists a cyclic representation $(\hc, \pi, \xi)$ of $\A$ and a unit vector $\eta\in\hc$ such that  $\dist(a,\B) = \bra \eta | \pi(a)\xi \ket$ and $\langle  \eta | \pi(b)\xi \rangle=0$ for all $b\in\B$.\end{theorem} 
\begin{proof} By the Hahn-Banach theorem,  there exists  $\psi\in \A^*$  such that $\|\psi\| = 1$, $\psi(a) = \dist(a,\B)$ and $\psi(b) = 0$ for all $b\in\B$.  By Lemma 3.3 of \cite{rieffel 1}, there exists a cyclic representation $(\hc, \pi, \xi)$ of $\A$ and a unit vector $\eta\in\hc$ such that $\psi(c) = \bra  \eta | \pi(c)\xi \ket \text{ for all }c\in\A.$
\end{proof}

It was shown in \cite{Bhatia} that for any $A\in \mathbb M_n(\bC)$ $$\dist(A,\bC 1_{\A}) = \max\{|\bra y | Ax \ket| : x, y\in\bC^n, \|x\|= \|y\| = 1 \text{ and } x\bot y\}.$$ Using Theorem \ref{8}, we obtain a similar formula for $\dist (a, \B)$, in the general case of a unital $C^*$-algebra $\A$ and $\bC 1_{\A}$ replaced with any subspace $\B$. We have
\begin{align}
\dist(a,\B)&= \max\left\{\bigg|\bra \eta | \pi(a) \xi \ket\bigg|: (\hc, \pi, \xi) \text{ is a cyclic representation of } \A, \eta\in\hc, \right.\\
&  \hspace{4cm} \|\eta\|=1 \text{ and }  \langle \eta | \pi(b)\xi\rangle=0 \text{ for all } b\in \B\bigg\}.\nonumber\end{align}

Under the restriction that best approximation to $a$ in $\B$ exists, the above formula was obtained in \cite[Theorem 4.3]{GroverSingla}. 
Another formula for $\dist (a, \B)$ when $\B$ is a $C^*$-subalgebra of $\A$ was proved in Theorem 3.2 of \cite{rieffel 1}. For more distance formulas, see \cite{Bhatia} and \cite{2014} for a discussion in  $\mathbb M_n(\bC)$, \cite{2012} and \cite{paul} for $\mathscr B(\hc)$ and \cite{2012} for general complex $C^*$-algebras and Hilbert $C^*$-modules over a complex $C^*$-algebra. 

A Hilbert $C^*$-module $\E$ over $\A$ is a right $\A$-module with a function $\left<\cdot,\cdot\right>:\E\times\E\rightarrow\A$, known as $\A$-valued semi-inner product, with the following properties for $\xi, \eta, \zeta\in\E, a\in\A, \lambda\in\bC :$
\begin{enumerate}
\item $\left<\xi,\eta+\zeta\right> = \left<\xi,\eta+\zeta\right> \text{ and } \left<\xi,\lambda\eta\right> = \lambda \left<\xi,\eta\right>$,
\item $\left<\xi,\eta a\right> =  \left<\xi,\eta\right> a$,
\item $\left<\xi,\eta\right> =  \left<\eta,\xi\right>^*$,
\item $\left<\xi,\xi\right>$ is a positive element of $\A$.
\end{enumerate}
Let $\hk$ be a Hilbert space. Let $\mathscr B(\hc, \hk)$ denotes the space of bounded $\bF$-linear operators from $\hc$ to $\hk$. It is a Hilbert $C^*$-module over $\mathscr B(\hc)$ with $\left<A,B\right> = A^*B$ for all $A, B\in \mathscr B(\hc, \hk)$. The below result extends Theorem 2.7 of \cite{2012} and Theorem 4.4 of \cite{2013}. 

\begin{theorem}\label{27} Let $e\in \E$. Let $\B$ be a subspace of $\E$. Then $e$ is Birkhoff-James orthogonal to $\B$ in the Banach space $\E$ if and only if there exists $\phi\in S_{\A}$ such that $\phi(\left<e,e\right>)= \|e\|^2 \mbox{ and } \phi(\left<e,b\right>) = 0$ for all $b\in\B$.
\end{theorem}\begin{proof} We prove the theorem for the special case $\E =  \mathscr B(\hc, \hk)$ . The general  case follows by Lemma 4.3 of \cite{2013}. 
The reverse direction is easy. Now let $e$ be orthogonal to $\B$. For any operator $t\in\mathscr B(\hc,\hk)$ we denote by $\tilde t$, the operator on $\hc\oplus \hk$ given by
 $ \tilde t
=\left[\begin{array}{clrr}
0 & 0 \\
t & 0
\end{array} \right].
$
We have $e$ is Birkhoff-James orthogonal to $\B$  if and only if $\tilde e$ is Birkhoff-James orthogonal to $\tilde\B = \{\tilde b : b\in\B\}$. Now using Corollary \ref{10}, we get that there exists $\tilde\phi\in\mathcal S_{\mathscr B(\hc\oplus \hk)}$ such that $\tilde\phi(\tilde e^*\tilde e) = \|\tilde e\|^2 \text{ and }\tilde\phi(\tilde e^*\tilde b) =0 \text{ for all }\tilde b\in\tilde\B.$ Now $\phi$ defined as $\phi(e) = \tilde\phi(\tilde e)$ is the required state.
\end{proof}

Another approach to prove the above theorem has been briefly discussed after Theorem 3.7 in \cite{GroverSingla}.
We also remark that some related results with restricted hypotheses for $\mathscr B(\hc)$ and $\mathscr B(\hc, \hk)$ have appeared recently in \cite{arxiv}. The results in this article are stronger in these spaces.

\section{Remarks}\label{remarks}

\begin{remark} For a complex  $C^*$-algebra $\A$ and a $C^*$-subalgebra $\B$ of $\A$ such that $1_{\A}\in\B$, a conditional expectation from $\A$ to $\B$  is a positive linear map $E$ of norm $1$ such that $E(1_{\A}) = 1_{\A}$ and  $E(b_1 a b_2) = b_1E(a)b_2$ for all $b_1, b_2\in\B$ and $a\in\A$. For any given conditional expectation $E$ from $\A$ to $\B$, we can define a $\B$-valued inner product on $\A$ given by $\bra a_1|a_2\ket_{E} = E(a_1^*a_2)$ (see \cite{rieffel}). So 
\begin{eqnarray*} \bra a-E(a)|a-E(a)\ket_{E} &=& E((a-E(a))^*(a-E(a)))\\
&=& E(a^*a) - E(E(a)^*a) - E(a^*E(a))+E(E(a)^*E(a))\\
&=& E(a^*a) - E(a)^*E(a) - E(a)^*E(a) + E(a)^*E(a)E(1_{\A})\\
&=& E(a^*a) - E(a)^*E(a).\end{eqnarray*} 
For $\phi\in\mathcal S_{\A}$, we  have 
\begin{equation}
\phi(\bra a-E(a)|a-E(a)\ket_{E}) = \phi(E(a^*a)) - \phi(E(a)^*E(a)).\label{cc}
\end{equation}


Since $a^*a \leq \|a\|^21_{\A}$ and $E(1_{\A}) = 1_{\A}$, we get $ \phi(E(a^*a)) \leq \|a\|^2$ . 
So \begin{equation}\label{56} \phi(E(a^*a)) - \phi(E(a)^*E(a)) \leq \|a\|^2. \end{equation} By \eqref{cc} and \eqref{56}, we obtain \begin{equation}\label{54} \phi(\bra a-E(a)|a-E(a)\ket_{E}) \leq \|a\|^2.\end{equation}
 
Now for $b\in \B$, \begin{equation}\label{53} \bra a-E(a)|a-E(a)\ket_{E} = \bra a-b-E(a-b)|a-b-E(a-b)\ket_{E}.\end{equation}
By \eqref{54} and \eqref{53}, we obtain $\phi(\bra a-E(a)|a-E(a)\ket_{E}) \leq \|a-b\|^2$ for all $b\in \B$, and so
$\phi\left(\bra a-E(a)|a-E(a)\ket_{E}\right) \leq \dist(a,\B)^2 $. Thus we obtain a lower bound for $\dist(a, \B)$ as follows:
\begin{equation}\label{98}\dist(a, \B)^2\geq  \sup\{\phi(E(a^*a)-E(a)^*E(a)): \phi\in\mathcal S_{\A}, E \text{ is a conditional expectation from } \A \text{ to } \B \},\end{equation} {(where $\sup(\emptyset) = -\infty$).}\end{remark}

\begin{remark} In the case $\B = \bC 1_{\A}$, equality holds in \eqref{98}. To see this, let $\langle a, \bC1_{\A}\rangle$ be the subspace generated by $a$ and $1_{\A}$.  Let $\lambda_0 1_{\A}$ be a best approximation to $a$ in $\bC1_{\A}$. We define $\tilde{E}: \langle a, \bC1_{\A}\rangle \rightarrow \bC1_{\A}$ as $\tilde{E}(a+\lambda 1_{\A}) =(\lambda_0+\lambda)1_{\A}$. For any $c\in \A$, the norm of the best approximation of $c$ to $\bC 1_{\A}$ is less than or equal to $\|c\|$. Since $(\lambda_0+\lambda)1_{\A}$ is the best approximation to $a+\lambda 1_{\A}$, we get that $\|\tilde{E}\|=1$. By Hahn-Banach theorem, there exists an extension $E$ of $\tilde{E}$ which is of  norm $1$. By Corollary II.6.10.3 of \cite{Blackadar}, $E$ is a conditional expectation. By Theorem \ref{58}, there exists $\phi\in\mathcal S_{\A}$ such that $\dist(a, \B)^2 = \phi(a^*a) - |\lambda_0|^2$ and $\phi(a) =\lambda_0 = \phi(E(a))$. Since  $\phi\circ E = \phi$, we get the required state for which equality in \eqref{98} holds.  \end{remark}
\begin{remark} It would be very interesting to find a counterexample to equality in \eqref{98} when $\B \neq \bC 1_{\A}$. \end{remark}

\textbf{Acknowledgments}

We would like to thank  Sneh Lata and Ved Prakash Gupta for many useful discussions. We would also like to acknowledge several discussions with  Amber Habib, which helped us to understand the geometric ideas behind the theorems.
The research of the first-named author is supported by INSPIRE Faculty Award IFA14-MA-52 of DST, India, and by Early Career Research Award ECR/2018/001784 of SERB, India.

\end{document}